\documentclass{aims}
\usepackage{amsmath}
\usepackage{paralist}
\usepackage{graphics} 
\usepackage{epsfig} 
\usepackage{graphicx}  
\usepackage{epstopdf}
 \usepackage[colorlinks=true]{hyperref}
\hypersetup{urlcolor=blue, citecolor=red}

\def\currentvolume{X}
 \def\currentissue{X}
  \def\currentyear{200X}
   \def\currentmonth{XX}
    \def\ppages{X--XX}
     \def\DOI{10.3934/xx.xx.xx.xx}

\usepackage{comment}
\newtheorem{theorem}{Theorem}[section]

\newtheorem{proposition}{Proposition}

\theoremstyle{definition}

\newcommand{\ip}{_{i+\frac{1}{2}}}

\newcommand{\N}{\mathbb{N}}

\newcommand {\vp} {\varphi}

\newcommand {\Chi} {{\bf \raise 2pt \hbox{$\chi$}} }
\newcommand {\dt}   {\Delta t}
\newcommand {\dx}   {\Delta x}

\newcommand {\f}   {\frac}
\newcommand {\p}   {\partial}

\newcommand {\ga}   {\left}
\newcommand {\dr}   {\right}

\newcommand{\bea} {\begin{array}{rl}}
\newcommand{\eea} {\end{array}}
\newcommand{\bepa}{\left\{ \begin{array}{l}}
\newcommand{\eepa} {\end{array}\right.}

\usepackage{cleveref}
\usepackage{subcaption}

\title[Energy and implicit discretization of FP and KS equations] 
      {Energy and implicit discretization of the Fokker-Planck and Keller-Segel type equations}

\author[L. Almeida F. Bubba B. Perthame and C. Pouchol]{}


 \email{almeida@ljll.math.upmc.fr}
 \email{federica.bubba@upmc.fr}
 \email{benoit.perthame@upmc.fr}
 \email{pouchol.camille@gmail.com}

\thanks{$^*$ Corresponding author: Federica Bubba}

\begin{document}
\maketitle

\centerline{\scshape Luis Almeida Federica Bubba$^*$ Beno\^it Perthame and Camille Pouchol}
\medskip
{\footnotesize
 \centerline{Sorbonne Universit\'{e}, CNRS, Universit\'{e} Paris-Diderot SPC,}
   \centerline{Inria, Laboratoire Jacques-Louis Lions,}
   \centerline{4, pl. Jussieu 75005, Paris, France}
} 

\bigskip


\begin{abstract}
The parabolic-elliptic Keller-Segel equation with sensitivity saturation, because of its pattern formation ability, is a challenge for numerical simulations. We provide two finite-volume schemes that are shown to preserve, at the discrete level, the fundamental properties of the solutions, namely energy dissipation, steady states, positivity and conservation of total mass. These requirements happen to be critical when it comes to distinguishing between discrete steady states, Turing unstable transient states, numerical artifacts or approximate steady states as obtained by a simple upwind approach.
\par
These schemes are obtained either by following closely the gradient flow structure or by a proper exponential rewriting inspired by the Scharfetter-Gummel discretization. An interesting fact is that upwind is also necessary for all the expected properties to be preserved at the semi-discrete level. 
These schemes are extended to the fully discrete level and this leads us to tune precisely the terms according to explicit or implicit discretizations. 
Using some appropriate monotonicity properties (reminiscent of the maximum principle), we prove well-posedness for the scheme as well as all the other requirements. Numerical implementations and simulations illustrate the respective advantages of the three methods we compare.
\end{abstract}

\section{Introduction}
\label{sec:intro}

Taxis-diffusion and aggregation equations are widely studied in the context of biological populations (see~\cite{Carrillo2018, Hillen2009, Hillen2007, Murray2002} for instance). They describe cell communities which react to external stimuli and form aggregates of organisms (pattern formation), such as bacterial colonies, slime mold or cancer cells. The Patlak-Keller-Segel model~\cite{KellerSegel1970} is the most famous system and we are interested in the following generalization
\begin{align}\label{eq:u}
\begin{cases}
\f{\p u}{\p t} - \f{\p}{\p x}\left[ \f{\p u }{\p x} - \vp(u) \f{\p v}{\p x}\right] = 0, \qquad &x\in (0,1), \, t > 0,
\\[1mm]
\f{\p u }{\p x} - \vp(u)  \f{\p v}{\p x} =0, \qquad &\text{for $x=0$ or $1$},
\\[1mm]
u(x,0)=u^0(x) \geq 0,\qquad &x\in [0,1].
\end{cases}
\end{align}
Here, $u(x,t) \geq 0$ represents the density of a given quantity (e.g. cells or bacteria population) and the initial data $u^0(x)$ is a given nonnegative smooth function. As for the function $v$, which models a molecular concentration, we choose either the case of the Fokker-Planck (FP in short) equation, where $v(x)$ is known
\begin{equation}\label{eq:FP}
v:= v(x) \geq 0,   \qquad  \f{\p v}{\p x} \in L^\infty(0,1),
\end{equation}
or the case of the generalized Keller-Segel (GKS in short) equation, where
\begin{equation}\label{eq:KS}
v(x,t)=\int K(x,y) u(y,t) dy,  \qquad \hbox{$K(x,y)$ a smooth, symmetric kernel}.
\end{equation}

Depending on the modeling choice for $\varphi(u)$, solutions to~\ref{eq:u} can blow up in finite time depending upon a critical mass (see~\cite{Blanchet2006, Nagai1995}) or reach stationary profiles in the form of peaks or plateaus~\cite{Potapov2005} (pattern formation by Turing instability). 
The high nonlinearities due to the advection term make problem~\ref{eq:u} mainly untractable through analytical methods. Thus, it is important to provide reliable numerical methods avoiding non-physical oscillations and numerical instabilities even when dealing with non-smooth solutions. The main properties that one wishes to preserve in a numerical method are
\begin{itemize}
	\item[(P1)] {\em positivity property}, since we are dealing with densities or concentrations,
	\begin{equation}
	\label{prop1}
	u(x,t) \geq 0,
	\end{equation}
	\item[(P2)] {\em mass conservation}, because no-flux boundary conditions are imposed,
	\begin{equation}\label{prop2}
	\int_0^1 u(x,t) = \int_0^1 u^0(x) dx,
	\end{equation}
	\item[(P3)] {\em preservation of discretized steady states} of the form
	\begin{equation}\label{prop3}
	g(u)= \mu + v, \qquad g'(u) = \f{1}{\vp(u)},
	\end{equation}
	where $\mu$ is a parameter related to the mass of $u$, and
	\item[(P4)] {\em energy dissipation} 
	\begin{equation}\label{prop4}
	\f{d}{dt} {\mathcal E}(t) \leq 0, \qquad \quad {\mathcal E}(t)=  \int_0^1 \left[G(u) - \kappa u v \right] dx,
	\end{equation}
\end{itemize}
where $G(u)$ is a primitive of $g(u)$ and the value of $\kappa$ differs for the two cases we study here, namely 
\begin{equation}\label{energy}
\kappa = 1 \quad \text{(FP case)}, \qquad \quad   \kappa = \f 12 \quad \text{(GKS case)}.
\end{equation}

The aims of our work are first to recall two points of view for the derivation of the above energy inequality, second to use them for the construction of conservative, finite volume numerical schemes preserving energy dissipation to solve equation~\ref{eq:u}, third to make numerical comparisons in the case of complex patterns in order to distinguish physical instabilities from numerical artifacts. 
The two different derivations of the energy dissipation use two symmetrization strategies: the gradient flow or the Scharfetter-Gummel approach.  
It turns out that they lead to two strategies for discretization of problem~\ref{eq:u}. We prove that the proposed schemes statisfy properties~\ref{prop1}--\ref{prop4} and because we build implicit schemes, there is no limitation on the time step in the fully discrete case. 

There exist other works which propose schemes for the resolution of problems in the form~\ref{eq:u}. For instance, finite elements methods are used, see~\cite{Saito2005} and references therein. Optimal transportation schemes for Keller-Segel systems are introduced in \cite{BlanchetCarrilloCalvez2008}. The papers~\cite{Carrillo2015} and~\cite{Carrillo2017} propose a finite-volume method able to preserve the above properties, including energy dissipation, at the semi-discrete level or with an explicit in time discretization, using the gradient flow approach, see also~\cite{Filbet2012}. The symmetrization using the Scharfetter-Gummel approach is used in \cite{LiuWangZhou2018} where properties similar to ours are proved. However, the results do not include sensitivity saturation.
To the best of our knowledge, our work is the first to propose implicit in time methods, without time step limitation (CFL condition), for which we are able to prove that, under generic conditions, the energy decreases at both semi-discrete and discrete level. Moreover, we build an alternative to the gradient flow approach applying the Scharfetter-Gummel strategy~\cite{Scharfetter1969} for the discretization of drift-diffusion equations~\ref{eq:u} with a general saturation function~$\varphi$.
\\

The paper is organized as follows.
In~\Cref{sec:model}, we present in more details our assumptions for the equation~\ref{eq:u}. We also explain some modeling choices in particular for the nonlinearity $\vp(u)$ and on the choice of the kernel $K$.
Section \ref{sec:energy} is devoted to the introduction of the two approaches, gradient flow or Scharfetter-Gummel, and to how we use the continuous version of energy dissipation to derive the schemes.
In~\Cref{sec:semidiscret} and~\Cref{sec:discret}, we show how the aforementioned two approaches lead to two different numerical methods, developed from the semi-discrete (only in space discretization) level to the fully discretized scheme. In particular, using a general result recalled in \Cref{sec:existunique} about monotone schemes, we  prove that the proposed schemes are well-posed and satisfy the fundamental properties~\ref{prop1}--\ref{prop4}. These theoretical results are illustrated in Section~\ref{sec:numericalsimulations} by numerical simulations: we compare the gradient flow and the Scharfetter-Gummel schemes with the upwind approach, typically used to solve this kind of models.

\section{Modeling assumptions}
\label{sec:model}
The standard biological interpretation of~\ref{eq:u} (\cite{Hillen2002, Murray2002, Perthame2000}) provides us with some further properties of the nonlinearities which we describe now.

\subsection*{Chemotactic sensitivity.}
The function $\vp(u)$ is called chemotactic sensitivity. It determines how the random movement of particles of density $u$ is biased in the direction of the gradient of $v$. In order to include the different choices of $\vp$, as $\vp(u) = u$ as in the Keller-Segel or drift-diffusion model, or the logistic case $\vp(u)=u(1-u)$, or the generalized case $\vp(u) = u e^{-u}$, we use the formalism
\begin{equation}\label{LogisticSensitivity}
\vp(u) = u \psi(u), \qquad \text{with} \qquad \psi(u)\geq 0, \qquad  \psi'(u)\leq 0.
\end{equation}
More precisely, we consider two cases for the smooth function $\psi$,  
\begin{equation}\label{AlwaysPositive}
\psi(u)>0, \qquad \forall u>0,
\end{equation}
or
\begin{equation}\label{Logistic}
\psi(u)>0 \quad \text{for} \quad  0<u<M, \qquad \psi(M)=0.
\end{equation}
In the case \ref{Logistic} we only consider solutions which satisfy $u\in [0,M]$.

It is convenient to introduce the notations
\begin{equation}\label{gG}
g(u) =\int_a^u \f{1}{\vp(v)} dv,  \qquad  G(u) =\int_0^u g(s) \, ds,
\end{equation}
where $a$ is a constant chosen so that $g$ is well-defined (depending on $\varphi$).
For instance, for the standard case $\vp(u)=u$, $a=1$ and one obtains $g(u)= \ln (u)$ and $G(u)=u\ln(u) -u$. For functions $\varphi$ satisfying \ref{AlwaysPositive}, a natural hypothesis which is related to blow-up is the following
\begin{equation}\label{BlowUp}
\frac{1}{\varphi} \notin L^1(a,+\infty), \qquad g(u)  \underset{ u \to \infty }{\longrightarrow} + \infty,
\end{equation}
an assumption which, as we see it later, appears naturally when it comes to the well-posedness of numerical schemes.

Note that under assumption~\ref{AlwaysPositive} and if $\psi$ is bounded, solutions exist globally and are uniformly bounded~\cite{Cieslak2007}. Under assumption~\ref{Logistic}, if $0\leq u^0 \leq M$, the solution is also globally defined and satisfies $0\leq u(t,\cdot) \leq M$ for all times~\cite{Hillen2007}.

\subsection*{Expression of the drift $v$.} The convolution expression for $v$ as a function of $u$ has been widely used in recent studies~\cite{Armstrong2006, Barre2018, Bertozzi2009}. It also comes from the Keller-Segel model~\cite{Hillen2009, KellerSegel1970, Perthame2000}, where the equation for the cells density in~\ref{eq:u} is complemented with a parabolic equation for the chemoattractant concentration $v$. Since the chemoattractant is supposed to diffuse much quicker than the cells density, we can consider a simplified form of the Keller-Segel system and couple~\ref{eq:u} with the elliptic equation for $v$
\begin{align*}
\begin{cases}
- \f{\p^2 v }{\p x^2} = u - v, &\qquad x \in (0,1), \\[1mm]
\f{\p v}{\p x} = 0, &\qquad x = 0 \text{ or } 1.
\end{cases}
\end{align*}
This equation leads to~\ref{eq:KS} using the Green function given by the positive and symmetric kernel $K(x,y)$ defined as
\begin{equation}
K(x,y) = \lambda \ga(e^x + e^{-x}\dr)  \left(e^y + e^{2-y} \right)\label{Kernel}, \quad x \leq y, \qquad \quad  \lambda = \f{1}{2\left(e^2-1\right)}.
\end{equation}

\section{Energy dissipation}
\label{sec:energy}

Energy dissipation is the most difficult property to preserve in a discretization and methods might require corrections \cite{JinYan2011}. Therefore, it is useful to recall how it can be derived simply for the continuous equation. We focus on two different strategies, that lead to two different discretization approaches, the gradient flow approach and the Scharfetter-Gummel approach.

\subsection{The gradient flow approach to energy.} Using the notations~\ref{gG}, the equation for $u$ can be rewritten as
\begin{equation}\label{gf}
\f{\p u}{\p t} - \f{\p}{\p x}\left[ \vp(u) \f{\p \left(g(u) - v \right)}{\p x} \right] = 0,
\end{equation}
so that
\begin{equation*}
\begin{split}
\left(g(u) -v\right) \f{\p u}{\p t}  =& \ga(g (u) -v\dr) \f{\p}{\p x}\left[ \vp(u) \f{\p \left(g(u) -v \right) }{\p x} \right] \\ 
=& \f 12 \f{\p}{\p x}\left[\vp(u) \f{\p (g (u )-v)^2}{\p x} \right] - \vp(u) \left[ \f{\p (g (u )-v)}{\p x}\right]^2.
\end{split}
\end{equation*}
Consequently, we find, in the Fokker-Planck case
\begin{align*}
\f{d}{dt} \int_0^1 [G(u) -uv(x)]dx = - \int_0^1\vp(u)  \ga[ \f{\p ( g(u )-v)}{\p x}\dr]^2 dx \leq 0,
\end{align*}
and in the generalized Keller-Segel case 
\begin{align*}
\f{d}{dt} \int_0^1 [G(u) - \f 12 uv(x,t)]dx = - \int_0^1 \vp(u)  \ga[ \f{\p (g (u)-v)}{\p x}\dr]^2 dx \leq 0,
\end{align*}
because, thanks to the symmetry assumption on $K$ and by using the definition~\ref{eq:KS} of $v$, we have
\begin{equation*}
\begin{split}
\int_0^1 \int_0^1  K(x,y) u(y,t) \f{\p u(x,t)}{\p t} =&  \int_0^1 \int_0^1  K(x,y)  \f{\p u(y,t)}{\p t}u(x,t) \\
=& \f 12  \f{d}{d t} \int_0^1 \int_0^1 K(x,y) u(y,t) u(x,t).
\end{split}
\end{equation*}

\subsection{The Scharfetter-Gummel approach to energy.} Inspired from the case of electric forces in semi-conductors, the 
equation for $u$ can be rewritten as
\begin{equation}\label{sg}
\f{\p u}{\p t}  - \f{\p}{\p x}\ga[ e^{v-g(u)}  \vp(u) \f{\p e^{g(u)-v}}{\p x} \dr] = 0,
\end{equation}
so that
\begin{equation*}
\begin{split}
(g(u)-v) \f{\p u}{\p t} =& (g(u)-v) \f{\p}{\p x}\left[ e^{v-g(u)}  \vp(u) \f{\p e^{g(u)-v}}{\p x} \right]\\
=& \f{\p}{\p x} \left[  (g(u)-v)e^{v-g(u)}  \vp(u) \f{\p e^{g(u)-v}}{\p x} \right]\\
&  - e^{v-g(u)} \vp(u) \f{\p e^{g(u)-v}}{\p x} \f{\p \left(g(u)-v\right)}{\p x}.
\end{split}
\end{equation*}
It is immediate to see that the last term has the negative sign while the time derivative term is exactly the same as in the gradient flow approach.
\\

At the continuous level, these two calculations are very close to each other. However, they lead to the construction of different discretizations. The gradient flow point of view is used for numerical schemes by~\cite{Carrillo2018}, the Scharfetter-Gummel approach is used in~\cite{LiuWangZhou2018}.

\section{Semi-discretization}
\label{sec:semidiscret}

We give here our notations for the semi-discretization. We consider a (small) space discretization $\dx= \frac{1}{I}$, $I \in \N$. The mesh is centered at  $x_i= i \dx$ for $i=1, \ldots, I$, with endpoints $x_{i+1/2} = (i+1/2) \dx$ for $i=1,\ldots, I$. Therefore, our computational domain is always shifted and takes the form $\big(\f{\dx}{2},(I+\f 12) \dx \big)$. Finally, the mesh is formed by the intervals
\begin{align*}
I_i= \left(x_{i-\f12}, x_{i+\f12} \right), \qquad  i=1, \ldots, I.
\end{align*}
The semi-discrete approximation of $u(x,t)$ at a given time $t$, interpreted in the finite volume sense (\cite{Bouchut2004, Eymard2000, LeVeque2002}), is denoted by
\begin{align*}
u_i(t) \approx  \f{1}{\dx}\int_{I_i} u(x, t) dx, \qquad  i=1, \ldots, I.
\end{align*}
As for the discretization of $v$ for $i=1, \ldots, I$, $v_i(t)$ stands for $v(x_i)$ in the FP case, while it is defined by $\sum_{j=1}^I K_{ij} u_j(t)$ in the GKS case, where \[K_{ij} := K(x_i, x_j), \qquad i=1, \ldots, I, \; j=1 \ldots, I.\]

Integration on the interval $I_i$ yields fluxes $F_{i+\f 12}(t)$ which will approximate the quantity $\left(\f{\p u }{\p x} - \vp(u) \f{\p v}{\p x}\right)$ at $x_{i+\f12}$ for $i=0, \ldots, I$ through the interval interfaces. The boundary conditions require $F_{\f 12}(t) = F_{I+\f 12}(t) = 0$, and the whole problem will be to properly define $F_{i+\f 12}(t)$ for $i=1, \ldots, I-1$. This will be chosen depending on the strategy, either through a gradient flow or a Scharfetter-Gummel approach, as well as the equation considered (FP or GKS).
\\

The mass conservative form of~\ref{eq:u} leads to a finite volume semi-discrete scheme 
\begin{equation}\label{eq:sd}
\bepa
\f{d u_i (t)}{dt} +\f 1 {\dx} [F_{i+1/2} (t) -F_{i-1/2}(t) ] =0, \qquad i=1,\dots,I, \quad t > 0,\\[1mm]
F_{1/2}(t) = F_{I+1/2}(t) = 0.
\eepa
\end{equation}

\noindent We use the definition~\ref{energy} for $\kappa$ and set for $i = 1, \dots, I$
\begin{align*}
E_i(t)=  G(u_i(t)) -  \kappa u_i(t) v_i(t).
\end{align*}

\noindent The semi-discrete energy is then
\begin{equation*}
{\mathcal E}_{\rm sd}(t) := \dx\sum_{i=1}^{I}  E_i (t).
\end{equation*}

\subsection{The gradient flow approach}
\label{sec:sdgl}

Using the form~\ref{gf} of equation~\ref{eq:u}, we define the semi-discrete flux as
\begin{equation}\label{sd:GF}
F_{i+1/2}(t) = - \f{\vp_{i+1/2}}{\dx} \big[ g(u_{i+1}) - v_{i+1} -   \big( g(u_{i} )- v_{i} \big)  \big], \quad i=1, \ldots, I-1.
\end{equation}
The precise expression of $\vp_{i+1/2}$ is not relevant for our present purpose which is to preserve the energy dissipation property. However, for stability considerations it is useful to upwind, an issue which we shall tackle when we consider the full discretization.
\\
Then, the semi-discrete energy form is obtained after multiplication by $(g (u_i) -v_i)$ and yields
\begin{equation*}
\begin{split}
\f{d}{dt}  \dx\sum_{i=1}^I  E_i (t) =& - \sum_{i=1}^{I} (g(u_i) - v_i)  [F_{i+1/2} -F_{i-1/2} ]\\
=& \sum_{i=1}^{I-1} F_{i+1/2} \left[ \left( g(u_{i+1}) -v_{i+1} \right) - ( g(u_i) -v_i)  \right].
\end{split}
\end{equation*}
Therefore, we find the semi-discrete form of energy dissipation
\begin{equation*}
\f{d {\mathcal E}_{\rm sd} }{dt} = - \dx \sum_{i=1}^{I-1} \vp_{i+1/2} \left[ \f{ \big( g(u_{i+1}) - v_{i+1}\big) -   \big( g(u_{i}) - v_{i} \big) }{\dx} \right]^2 \leq 0.
\end{equation*}

\subsection{The Scharfetter-Gummel approach}
\label{sec:sdsg}

We choose to discretize the form~\ref{sg}, defining the semi-discrete flux as
\begin{equation}\label{sd:SG}
F_{i+1/2} (t) = - \f{\big( e^{v-g(u)}  \vp(u) \big)_{i+1/2} }{\dx} \big[  e^{g(u_{i+1}) - v_{i+1}} - e^{g(u_{i} ) - v_{i}}  \big], \quad i=1, \ldots, I-1,
\end{equation}
where, again, the specific form of the interpolant $\big( e^{v-g(u)}  \vp(u) \big)_{i+1/2}$ is not relevant here.

As above, the semi-discrete energy form follows upon multiplication by $g (u_i)-v_i $ and reads
\begin{equation*}
\begin{split}
\f{d}{dt}  \dx\sum_{i=1}^{I} E_i (t) =& - \sum_{i=1}^{I} \left(g (u_i )- v_i \right)  \left[F_{i+1/2} -F_{i-1/2} \right]\\
=& \sum_{i=1}^{I-1} F_{i+1/2} \left[ \left(g(u_{i+1} ) - v_{i+1} \right) - \left( g(u_i) -v_i \right) \right].
\end{split}
\end{equation*}
Summing up, the semi-discrete form of energy dissipation here writes
\begin{equation*}
\begin{split}
\f{d {\mathcal E}_{\rm sd} }{dt} =& - \dx \sum_{i=1}^{I-1} \Big{\{} \left( e^{v-g(u)}  \vp(u) \right)_{i+1/2} \cdot \f{ e^{g(u_{i+1}) - v_{i+1}} - e^{g(u_{i} ) - v_{i}} }{\dx} \cdot\\
& \qquad \qquad \qquad \f{ \left( g(u_{i+1} ) - v_{i+1} \right) - \left( g(u_i) -v_i \right)  }{\dx} \; \Big{\}} \leq 0.
\end{split}
\end{equation*}

\subsection{Discrete steady states}
\label{sec:dss}

Steady states make the energy dissipation vanish which imposes both in the gradient flow and the Scharfetter-Gummel approaches that
$\big( g(u_{i+1}) - v_{i+1}\big) =  \big( g(u_{i}) - v_{i} \big) $. In other words they are given, up to a constant $\mu$, as the discrete version of~\ref{prop3}, 
\begin{equation}\label{eq:discretestst}
g(u_{i}) = v_{i} + \mu, \quad  i=1,\dots,I.
\end{equation}
We recall from~\cite{Potapov2005} that in the GKS case, there are several steady states and the constant ones can be unstable.

\section{Fully discrete schemes}
\label{sec:discret}

For the time discretization, we consider (small) time steps $\dt>0$, and set $t^n =n \dt$ for $n \in \mathbb{N}$. The discrete approximation of $u(x,t)$ is now
\begin{equation*}
u^n_i \approx  \f{1}{\dx} \int_{I_i} u(x, t^n) dx, \qquad  i=1, \ldots, I, \; n \in \mathbb{N}.
\end{equation*}
Integration on the interval $I_i$ yields fluxes $F_{i+\f 12}^n$ which will approximate the quantity $\left(\f{\p u }{\p x} - \vp(u) \f{\p v}{\p x}\right)$ at $x_{i+\f12}$ for $i=1, \ldots, I-1$ through the interval interfaces, and at time $t^n$ for $n\in \mathbb{N}$. The boundary conditions require $F_{\f 12}^n = F_{I+\f 12}^n = 0$ for all $n \in \mathbb{N}$.

To achieve the time discretization, and restricting our analysis to the Euler scheme, we write the time discretization $\f{du_i(t)}{dt}$ as $ \f{u_i^{n+1} - u_i^{n}}{\dt}$. Therefore, the full discretization of~\ref{eq:sd} reads for $n \in \mathbb{N}$ as
\begin{equation}\label{eq:fd}
\bepa
u_i^{n+1} - u_i^{n} + \f{\dt}{\dx} \left[ F_{i+1/2}^{n+1} - F_{i-1/2}^{n+1} \right]  =0, \qquad i = 1, \dots, I,\\[1mm]
F_{1/2}^{n+1} = F_{I+1/2}^{n+1} = 0.
\eepa
\end{equation}

The issue here is to decide which terms (in $u$ and $v$) should be discretized with implicit or explicit schemes based on fully discrete energy dissipation. We claim that, apart from the interpolant, we need to make the terms in $u_i$ implicit and, for the GKS case, the terms in $v_i$ explicit, a fact on which we now elaborate. 

We define the energy at the discrete level through  
\begin{align*}
E_i^n=  G(u_i^n) -  \kappa u_i^n v_i^n, \qquad i=1,\dots,I, \; n \in \mathbb{N},
\end{align*}
and
\begin{equation*}
{\mathcal E}^n := \dx\sum_{i=1}^{I}  E_i^n, \qquad n \in \mathbb{N}.
\end{equation*}

The computation made in the semi-discrete case, $\f{dE_i (t) }{dt} =  \f{du_i (t) }{dt} (g(u_i(t)) - v_i(t))$, extends to the fully discrete setting and leads to the following constraint on the energy
\begin{equation*}
\sum_{i=1}^I \left(E_i^{n+1} -  E_i^{n}\right)  \leq \sum_{i=1}^I (u_i^{n+1} - u_i^{n}) (g( u_i^{\alpha_n}) - v_i^{\beta_n}).
\end{equation*}
Here, $u_i^{\alpha_n} := \alpha u_i^{n} + (1-\alpha) u_i^{n+1}$ and $v_i^{\beta_n} := \beta v_i^{n} + (1-\beta) v_i^{n+1}$.  
\noindent
The convexity of $G(\cdot)$ motivates the choice of an implicit discretization for $u$, i.e. $\alpha = 0$, because for $i=1, \ldots, I$, there holds
\begin{equation*}
G(u_i^{n+1}) -G(u_i^n) \leq g(u_i^{n+1})(u_i^{n+1} - u_i^{n}).
\end{equation*}
\noindent
Regarding the term in $uv$, only the case of GKS needs to be fixed and we thus require
\begin{align*}
- \sum_{i=1}^I \left[u_i^{n+1} v_i^{n+1} - u_i^n v_i^{n}\right] \leq -2 \sum_{i=1}^I v_i^{\beta_n} (u_i^{n+1} - u_i^{n}).
\end{align*}
It is natural to try and balance the terms by choosing a semi-explicit discretization with $\beta = \frac{1}{2}$, which yields
\begin{equation*}
\begin{split}
\sum_{i=1}^I 2 v_i^{\beta_n} (u_i^{n+1} - u_i^{n}) - \left(u_{i}^{n+1} v_i^{n+1} - u_{i}^{n} v_i^{n}\right) =& \sum_{i=1}^I \left(u_i^{n+1}v_i^n - u_i^n v_i^{n+1}\right)\\
=& \sum_{i,j} K_{ij} \left(u_i^{n+1}u_j^n - u_i^n u_j^{n+1}\right),
\end{split}
\end{equation*}
with the last term vanishing due to the symmetry of $K$.

However, implicit and explicit time discretizations for $v$ can also be considered at the expense of adding hypotheses on the kernel $K$. Indeed, for a given $0 \leq \beta \leq 1$, we find
\begin{equation*}
\sum_{i=1}^I 2 v_i^{\beta_n} (u_i^{n+1} - u_i^{n}) - \left(u_{i}^{n+1} v_i^{n+1} - u_{i}^{n} v_i^{n}\right) = (1-2\beta) \sum_{i,j} K_{ij} (u_{i}^{n+1} - u_{i}^n)(u_{j}^{n+1} - u_{j}^n).
\end{equation*}

As a consequence, an explicit (resp. implicit) scheme is suitable for the time discretization of $v$ provided that $K$ is a non-negative (resp. non-positive) symmetric kernel. Since $K$ is a non-negative symmetric kernel for the Generalized Keller-Segel equation~\ref{eq:KS}, for simplicity we choose an explicit discretization for $v$. 

Finally, we note that the interpolant does not play any role for energy discretization and we can use the simplest explicit or implicit discretization (both in $u$ and $v$), so as  to make the analysis of the scheme as simple as possible.

\subsection{The gradient flow approach}
\label{sec:dgl}

We consider the full discretization of~\ref{sd:GF} and define the fully discrete flux in~\ref{eq:fd} as
\begin{equation}\label{disc:gsflux}
F_{i+1/2}^{n+1} =- \f{\vp(u)_{i+1/2}^{n+1}}{\dx}  \left[  \left(g(u_{i+1}^{n+1}) - v_{i+1}^{n}\right)  -\left( g(u_{i}^{n+1}) - v_{i}^{n}\right) \right], \; \; i = 1, \dots, I-1.
\end{equation}

At this level, we need to define the form of the interpolant $\vp(u)_{i+1/2}^{n+1}$. From the theorem in~\Cref{sec:existunique}, we use an upwind technique in order to ensure well-posedness and monotonicity properties of the scheme. Thus, for $i=1,\dots,I-1$, we define
\begin{equation}\label{disc:gfupwind}
\vp(u)_{i+1/2}^{n+1} := 
\begin{cases}
u_{i}^{n+1} \psi(u_{i+1}^{n+1})  \quad \text{when} \quad g(u_i^{n+1}) - g(u_{i+1}^{n+1})+ v_{i+1}^{n} -  v_{i}^{n} \geq 0,\\[1mm]
u_{i+1}^{n+1} \psi(u_{i}^{n+1}) \quad \text{when} \quad   g(u_i^{n+1}) - g(u_{i+1}^{n+1}) + v_{i+1}^{n} -  v_{i}^{n} < 0.
\end{cases}
\end{equation}

\begin{proposition}[Fully discrete gradient flow scheme] We assume either~\ref{AlwaysPositive} and~\ref{BlowUp}, or~\ref{Logistic} and give the $u_i^0 \geq 0$. Then, the scheme~\ref{eq:fd}--\ref{disc:gsflux}--\ref{disc:gfupwind} has the following properties:
	\\[1mm]
	{\bf (i)} \, the solution $u_i^{n}$ exists and is unique, for all $i = 1, \dots, I$, and $n\geq 1$;
	\\[1mm]
	{\bf (ii)}\,  it satisfies  $u_i^{n} \geq 0$,  and  $u^n_i \leq M$ for  the case \ref{Logistic}, if it is initially true;
	\\[1mm]
	{\bf (iii)} the steady states $g(u_i) -v_i = \mu $ are preserved;
	\\[1mm]
	{\bf (iv)} the discrete energy dissipation inequality is satisfied
	\begin{equation*}
	{\mathcal E}^{n+1} -  {\mathcal E}^{n} \leq - \f{\dt}{\dx}  \sum_{i=1}^{I-1}  \vp(u)_{i+1/2}^{n}  \left[ \left(g(u_{i+1}^{n+1}) - v_{i+1}^{n}\right)  -\left(  g(u_{i}^{n+1}) - v_{i}^{n}\right) \right]^2.
	\end{equation*}
\end{proposition}

Notice that this theorem does not state a uniform bound in the case~\ref{AlwaysPositive} and~\ref{BlowUp}. 

\begin{proof}
{\bf(i)} We  prove that the scheme satisfies the hypotheses of the theorem in~\Cref{sec:existunique}. We set
\begin{align*}
A_{i+1/2}(u_i^{n+1}, u_{i+1}^{n+1}) = \f{\dt}{\dx} F_{i+1/2}^{n+1}.
\end{align*}
Then, the simplest case is when $\varphi$ satisfies~\ref{Logistic}, since clearly $u_i^{n+1} \equiv 0$ and $u_i^{n+1} \equiv M$ are respectively a sub- and supersolution. 
When $\varphi$ satisfies~\ref{AlwaysPositive} and~\ref{BlowUp}, $u_i^{n+1} \equiv 0$ is again a subsolution, while for the supersolution we choose $\bar{U}^{n+1}_i = g^{-1}(C +v_i^n)$. Such a choice indeed makes the flux terms vanish:
\begin{equation*}
\begin{split}
F_{i+1/2}^{n+1} =& - \f{\vp(u)_{i+1/2}^{n+1}}{\dx}  \left[ \left(g(\bar{U}_{i+1}^{n+1}) - v_{i+1}^{n}\right) - \left(  g(\bar{U}_{i}^{n+1}) - v_{i}^{n}\right) \right]\\ 
=& - \f{\vp(u)_{i+1/2}^{n+1}}{\dx}  \left[ C -C\right] = 0.
\end{split}
\end{equation*}
Thus $\bar{U}^{n+1}_i$ is a supersolution as soon as $g^{-1}(C + v_i^n) \geq u_i^n$, which holds when $C$ is taken to be large enough because we recall that assumption~\ref{BlowUp} ensures that $g(u)$ tends to $+\infty$ as $u$ tends to $+\infty$.\par

Moreover, the scheme is monotone since 
\begin{equation*}
\begin{split}
\partial_1 A \ip (u_i^{n+1}, u_{i+1}^{n+1}) = &- \frac{\Delta t}{(\Delta x) ^2} u_{i+1}^{n+1} \psi'(u_i^{n+1}) \left[ g(u_{i+1}^{n+1}) - v_{i+1}^n -\left( g(u_i^{n+1}) - v_i^n \right) \right]_{+}\\
&- \frac{\Delta t}{(\Delta x) ^2} \psi(u_{i+1}^{n+1}) \left[ g(u_{i+1}^{n+1}) - v_{i+1}^n -\left( g(u_i^{n+1}) - v_i^n \right) \right]_{-} \\
&- \frac{\Delta t}{(\Delta x) ^2} \varphi(u)^{n+1} \ip \left[- g'(u_i^{n+1})\right] \geq 0,
\end{split}
\end{equation*}
where
\begin{align*}
[x]_{+} = \begin{cases}x &\quad \text{for} \quad x \geq 0, \\ 0 &\quad \text{for} \quad x < 0 \end{cases}
\qquad \text{and} \qquad 
[x]_{-} = \begin{cases} 0 &\quad \text{for} \quad x \geq 0, \\ x &\quad \text{for} \quad x < 0, \end{cases} \; 
\end{align*}
so that $[x]_{+} \geq 0, \; [x]_{-} \leq 0$ for all $x$.  \\
{\bf (ii)} Positivity of discrete solutions and the upper bound in the logistic case follow from the subsolution and supersolution built in step (i).
\\
{\bf (iii)} Preservation of steady states at the discrete level follows immediately from the form we have chosen for the fully discrete fluxes.
\\
{\bf (iv)} For the energy inequality, we remark that the contribution regarding time discretization is treated in the introduction of the present section. The space term is exactly treated as in the corresponding subsection of Section~\ref{sec:semidiscret}.~\end{proof}

\subsection{The Scharfetter-Gummel approach}
\label{sec:dsg}

In~\ref{eq:fd}, the fully discrete Scharfetter-Gummel flux reads
\begin{equation*}
F_{i+1/2}^{n+1} = \big( e^{v^{n}-g(u^{n})}  \vp\big(u^{n+1}\big) \big)_{i+1/2} \left[   e^{g(u_{i+1}^{n+1}) - v_{i+1}^{n}} - e^{g(u_{i} ^{n+1}) - v_{i}^{n} }  \right],   \quad i=1,\ldots,I-1.
\end{equation*}
As for the gradient flow approach, we need the upwind technique to get a scheme which satisfies the hypotheses in~\Cref{sec:existunique}. So, we set for $i = 1, \dots, I-1$
\begin{equation*}
\left( e^{v^{n}-g(u^{n})}  \vp\left(u^{n+1}\right) \right)_{i+1/2} := 
\begin{cases}
&u_{i+1}^{n+1} \psi(u_i^{n+1}) e^{v_{i+1}^n - g(u_{i+1}^{n})}, \\ &\qquad \text{if} \quad e^{\left(g(u_{i+1}^{n+1}) - v_{i+1}^n \right)} -e^{\left( g(u_i^{n+1}) - v_i^n \right)} \geq 0, \\[1mm]
&u_{i}^{n+1} \psi(u_{i+1}^{n+1}) e^{v_{i}^n - g(u_{i}^{n})}, \\ &\qquad \text{if} \quad e^{\left(g(u_{i+1}^{n+1}) - v_{i+1}^n \right)} -e^{\left( g(u_i^{n+1}) - v_i^n \right)} < 0.
\end{cases}
\end{equation*}

\begin{proposition} [Fully discrete Scharfetter-Gummel scheme] We assume either~\ref{AlwaysPositive} and~\ref{BlowUp}, or~\ref{Logistic} and give the $u_i^0 \geq 0$. Then, the scheme~\ref{eq:fd}--\ref{disc:gsflux}--\ref{disc:gfupwind} has the following properties:
	\\[1mm]
	{\bf (i)} \, the solution $u_i^{n}$ exists and is unique, for all $i = 1, \dots, I$, and $n\geq 1$;
	\\[1mm]
	{\bf (ii)}\,  it satisfies $u_i^{n} \geq 0$, and $u^n_i \leq M$ for the case~\ref{Logistic}, if it is initially true;
	\\[1mm]
	{\bf (iii)} the steady states $g(u_i) -v_i = \mu $ are preserved;
	\\[1mm]
	{\bf (iv)} the discrete energy dissipation inequality is satisfied
	\begin{equation*}
	\begin{split}
	{\mathcal E}^{n+1} -  {\mathcal E}^{n} \leq & - \f{\dt}{\dx} 
	\sum_{i=1}^{I-1}  \Big{ \{ } \left( e^{v^{n}-g(u^{n})}  \vp\ga(u^{n}\dr) \right)_{i+1/2} \cdot \left[   e^{g\left(u_{i+1}^{n+1}\right) - v_{i+1}^{n}} - e^{g\ga(u_{i} ^{n+1}\dr) - v_{i}^{n} }  \right]\cdot \\[1mm]
	& \qquad \qquad \quad \left[  \left(g\left(u_{i+1}^{n+1}\right) - v_{i+1}^{n} \right) - \left( g\left( u_{i}^{n+1} \right) - v_{i}^{n}  \right)\right] \Big{\}}\leq 0.
	\end{split}
	\end{equation*}
\end{proposition}
\begin{proof}
We argue exactly as for the gradient flow approach.~\end{proof}

\subsection{The upwinding approach}
\label{sec:dup}

The upwind scheme is driven by simplicity and, in~\ref{eq:fd}, the fluxes are defined by
\begin{equation*}
F_{i+1/2}^{n+1} =- \frac{1}{\dx} \left[ u_{i+1}^{n+1} - u_{i}^{n+1} - \vp(u)_{i+1/2}^{n} \left(v_{i+1}^{n}  - v_{i}^{n}\right) \right],  \qquad i=1,\ldots,I-1, 
\end{equation*}
with
\begin{equation}
\vp(u)_{i+1/2}^{n+1} := 
\bepa
u_{i}^{n+1} \psi(u_{i+1}^{n+1})  \quad \text{when} \quad v_{i+1}^{n} -  v_{i}^{n} \geq 0,
\\[1mm]
u_{i+1}^{n+1} \psi(u_{i}^{n+1}) \quad \text{when} \quad   v_{i+1}^{n} -  v_{i}^{n} < 0,
\eepa
\label{disc:upwind}
\end{equation}
as in~\ref{disc:gfupwind}, but this time depending on the sign of $v_{i+1}^{n} -  v_{i}^{n}$.
\\
\begin{proposition}[Fully discrete upwind scheme] We assume either~\ref{AlwaysPositive} and~\ref{BlowUp}, or~\ref{Logistic} and give the $u_i^0 \geq 0$. Then, the scheme~\ref{eq:fd}--\ref{disc:gsflux}--\ref{disc:gfupwind} has the following properties:
	\\[1mm]
	{\bf (i)} \, the solution $u_i^{n}$  exists and is unique, for all $i = 1, \dots, I$,  and $n\geq 1$;
	\\[1mm]
	{\bf (ii)}\,  it satisfies  $u_i^{n} \geq 0$,  and  $u^n_i \leq M$ for  the case~\ref{Logistic}, if it is initially true.
\end{proposition}
\begin{proof}
As for the gradient flow approach, the above choice makes the scheme monotone, because
\begin{equation*}
\begin{split}
\frac{\dt}{\dx} \partial_1 F\ip (u_i^{n+1}, u_{i+1}^{n+1}) = & - \frac{\Delta t}{\Delta x ^2} \Big(- 1 - u_{i+1}^{n+1} \psi'(u_i^{n+1}) \left[ v_{i+1}^n - v_i^n \right]_{-}\\
 & \qquad \qquad - \psi(u_{i+1}^{n+1}) \left[v_{i+1}^n  - v_i^n \right]_{+} \Big) \geq 0.
\end{split}
\end{equation*}
Thus, arguing as for the gradient flow approach and relying on the results in Appendix~\ref{sec:existunique}, existence and uniqueness of the discrete solution as well as preservation of the initial bounds follow immediately.~\end{proof}

Thus, choice~\ref{disc:upwind} enables to prove that the scheme is well-defined, satisfies~$u_i^n \geq 0$ and preserves the bound $u_i^n \leq M$ for the case~\ref{Logistic}, but the energy dissipation inequality is lost. Also the steady states, in this case, are defined by the nonlinear relation $u_{i+1} - u_{i} = \vp(u)_{i+1/2} (v_{i+1}  - v_{i})$ which are usually not in the form~\ref{eq:discretestst}.

\section{Numerical simulations}
\label{sec:numericalsimulations}

\subsection{The Fokker-Planck equation, $\vp(u)=u$}
\label{sec:numFP}

We first present the numerical implementation of the Fokker-Planck equation with $\vp(u)=u$. We here compare the Scharfetter-Gummel and upwind approaches (the gradient flow gives the same solution as the Scharfetter-Gummel and is thus not presented here). Both these schemes have error convergence of order $1$ in space, as it can be easily checked.

We consider a first case with $\chi/ D = 24$, with $I=100$ and an initial density $u^0=1$. We take the velocity field as
\begin{equation*}
v=x(1-x) |x-0.3|.
\end{equation*}

\noindent In~\Cref{fig:fp_SG_UP}, we compare the approximate stationary solutions obtained with the upwind scheme (blue, dashed line) and the Scharfetter-Gummel scheme (red line) with the exact stationary solution (black line), which in this case has the form $u(x) = C e^{\chi v(x)/D}$, with $C = \left( \int_{0}^{1} e^{\chi v(x)/D} dx \right)^{-1}$ so that the mass of the stationary solution is normalized. In this first case, the two schemes have no significant differences; this is a major difference with the Keller-Segel equation, as we show it in the next subsection.
\begin{figure}[htp]
\begin{center}
	\includegraphics[width=0.4\linewidth]{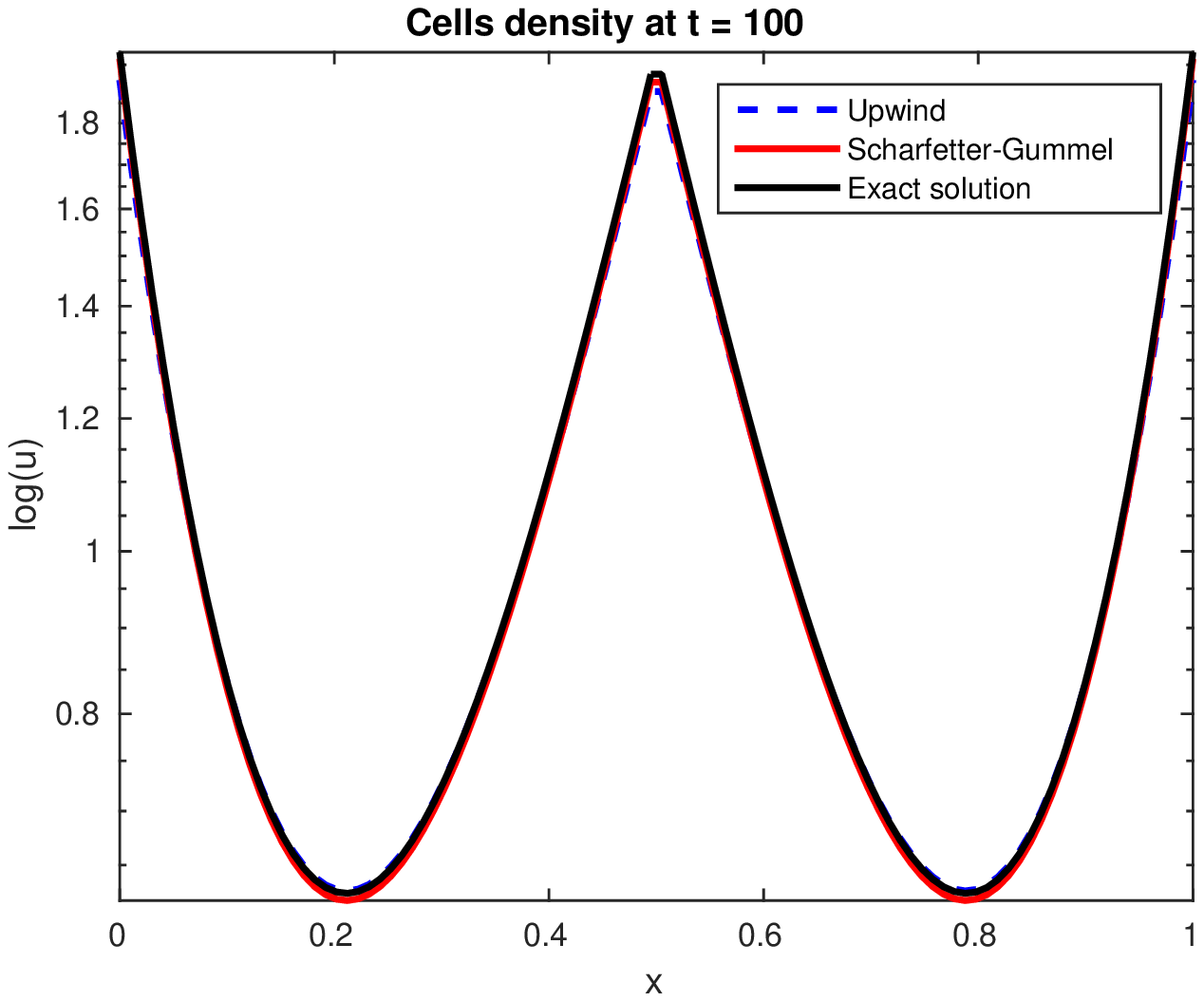}
	\includegraphics[width=0.4\linewidth]{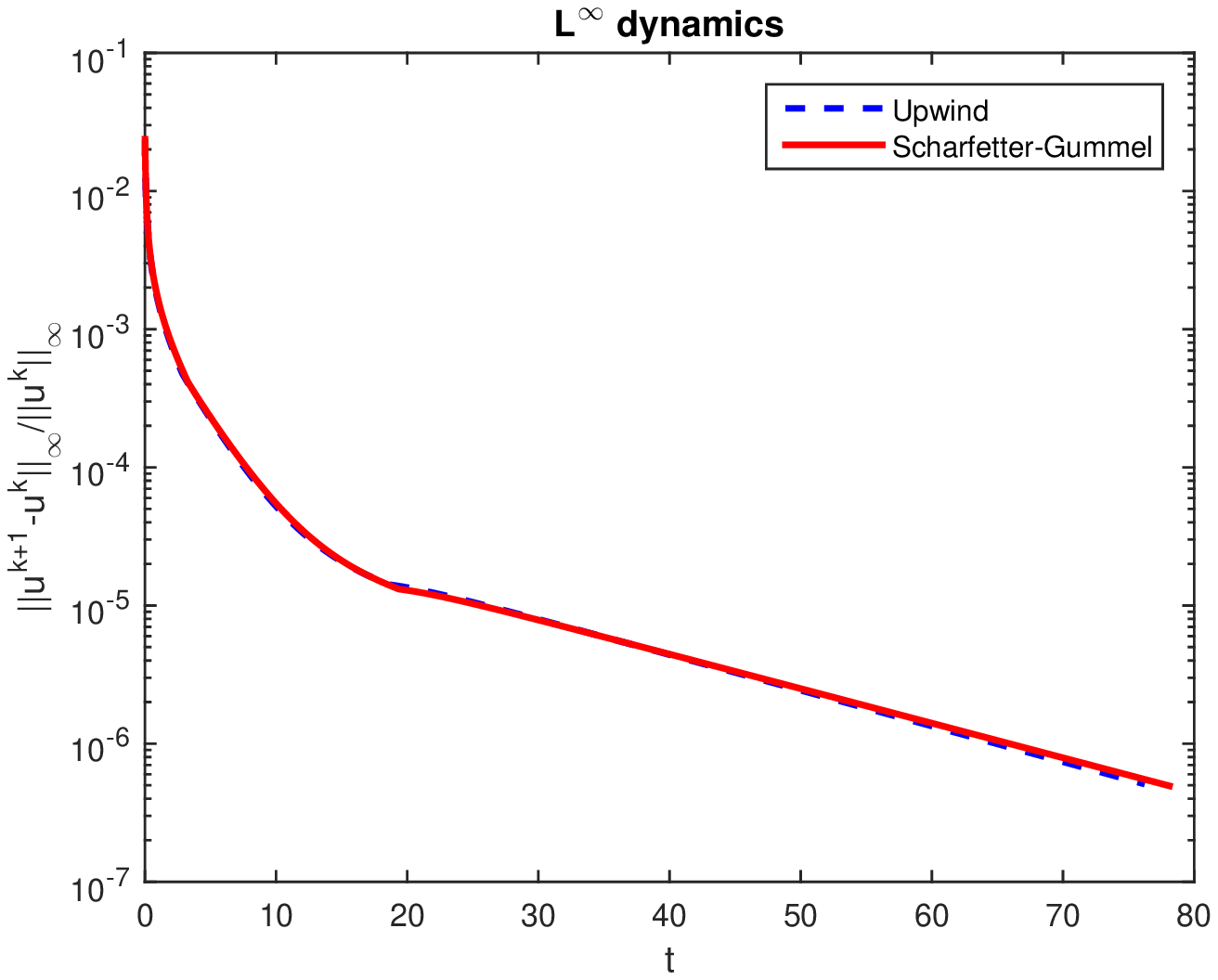}
	\caption{Left: Comparison of solutions of the Scharfetter-Gummel (red line) and upwind (blue, dashed line) schemes at time $t= 100$ with the exact stationary solution (black line) for the linear Fokker-Planck equation with $\vp(u)=u$. We used $I = 100$ and $\Delta t = 0.01$. Right: normalized $L^\infty$ variation for the two schemes.}
	\label{fig:fp_SG_UP}
\end{center}
\end{figure}

\subsection{The nonlinear Keller-Segel equation}
\label{sec:numKS}

We turn to the equation~\ref{eq:u} coupled with~\ref{eq:KS} for two nonlinear forms of the chemotactic sensitivity function: the logistic form $\vp(u) = u (1-u)$ and the exponential form $\vp(u) = u e^{-u}$. The goal is to compare the discrete solutions obtained with the three numerical approaches presented above when patterns arise, namely when Turing instabilities drive the formation of spatially inhomogeneous solutions (we refer to~\cite{Murray2002} for an introduction to this topic). To this end, we slightly modify the original equation~\ref{eq:u} to
\begin{align}\label{eq:ucoeff}
\begin{cases}
\f{\p u}{\p t} - \f{\p}{\p x}\left[ D \f{\p u }{\p x} - \chi \vp(u) \f{\p v}{\p x}\right] = 0, \qquad &x\in (0,1), \, t > 0,
\\[1mm]
D\f{\p u }{\p x} - \chi \vp(u)  \f{\p v}{\p x} =0, \qquad &\text{for $x=0$ or $1$},
\\[1mm]
u(x,0)=u^0(x) \geq 0,\qquad &x\in [0,1],
\end{cases}
\end{align}
in order to emphasize the coefficients driving the instabilities: $D>0$, the constant diffusion coefficient and $\chi>0$, the chemosensitivity. The concentration of the chemoattractant $v$ remains driven by~\eqref{eq:KS}.
\\
\begin{figure}[htp]
\begin{center}
	\includegraphics[width=0.7\linewidth]{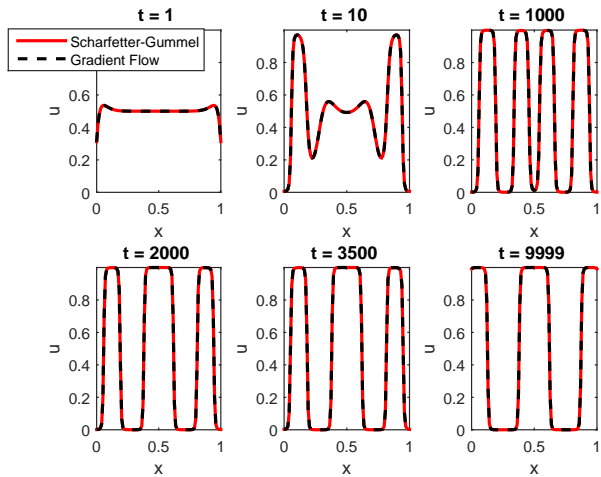}
	\caption{\textbf{Evolution in time of solutions} to \eqref{eq:ucoeff} in the logistic case $\vp(u) = u (1-u)$ with $\chi / D = 40$. We solved the equation with the Scharfetter-Gummel (red line) and the gradient flow scheme (black dashed line) with $I = 100$ and $\dt = 1$. There is no  major difference between the solutions given by the two schemes.}
	\label{fig:logistic_SG_GF_ev}
\end{center}
\end{figure}

We first consider the logistic case with $\chi / D = 40$. We take as initial condition a random spatial perturbation of the constant steady state $u^0 = 0.5$ and solve the equation with 100 equidistant points in $[0,1]$.

Figure~\ref{fig:logistic_SG_GF_ev} shows the evolution in time of the density $u_i^n$ obtained with the Scharfetter-Gummel (red line) and the gradient flow schemes (black, dashed line). After a rather short time, the initial spatial perturbation evolves, as expected, in spatially inhomogeneous patterns: a series of ``steps'' arise in the regions where the concentration of the chemoattractant is greater. After some time, a structure with a smaller number of steps forms when the two central plateaus merge. It is worth noticing that, even if the transitions from one structure to another happen very quickly, the time period during which these structures remain unchanged grows with the number of transitions that occurred. In~\cite{Potapov2005}, these intermediate patterns are called \textit{metastable}, and this peculiar phenomenon is explained in details.

\begin{figure}[htp]
\begin{center}
	\includegraphics[width=0.7\linewidth]{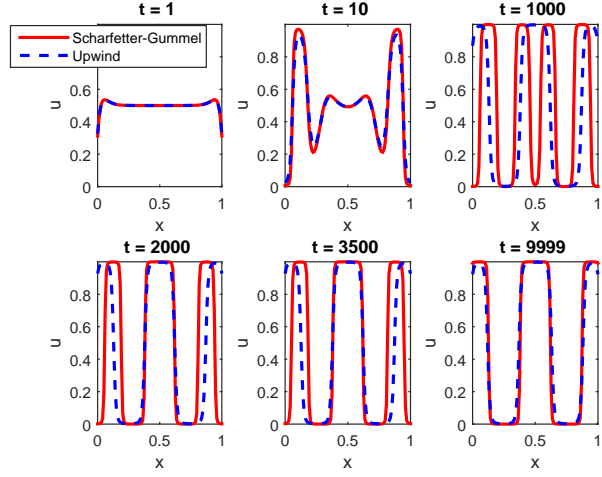}
	\caption{\textbf{Evolution in time of solutions} to~\ref{eq:ucoeff} in the logistic case $\vp(u) = u (1-u)$ with $\chi / D = 40$. We solved the equation with the Scharfetter-Gummel (red line) and the upwind scheme (blue, dashed line) with $I = 100$ and $\dt = 1$.}
	\label{fig:logistic_SG_UP_ev}
\end{center}
\end{figure}

\begin{figure}[htp]
	\centering
	\begin{subfigure}[b]{0.32\textwidth}
		\includegraphics[width=\textwidth]{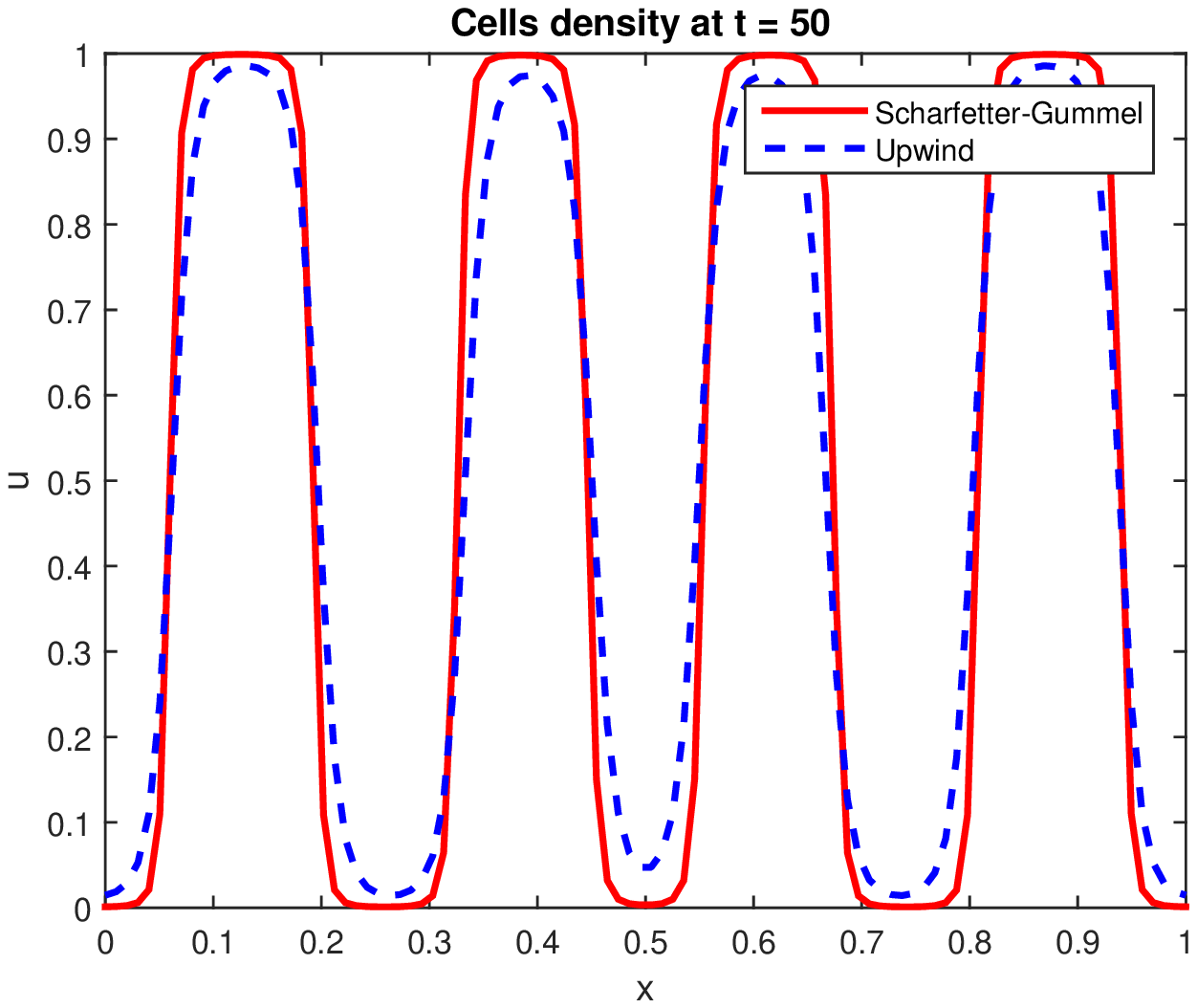}
		\caption{}
		\label{fig:logistic_SG_UP_50}
	\end{subfigure}
	\begin{subfigure}[b]{0.32\textwidth}
		\includegraphics[width=\textwidth]{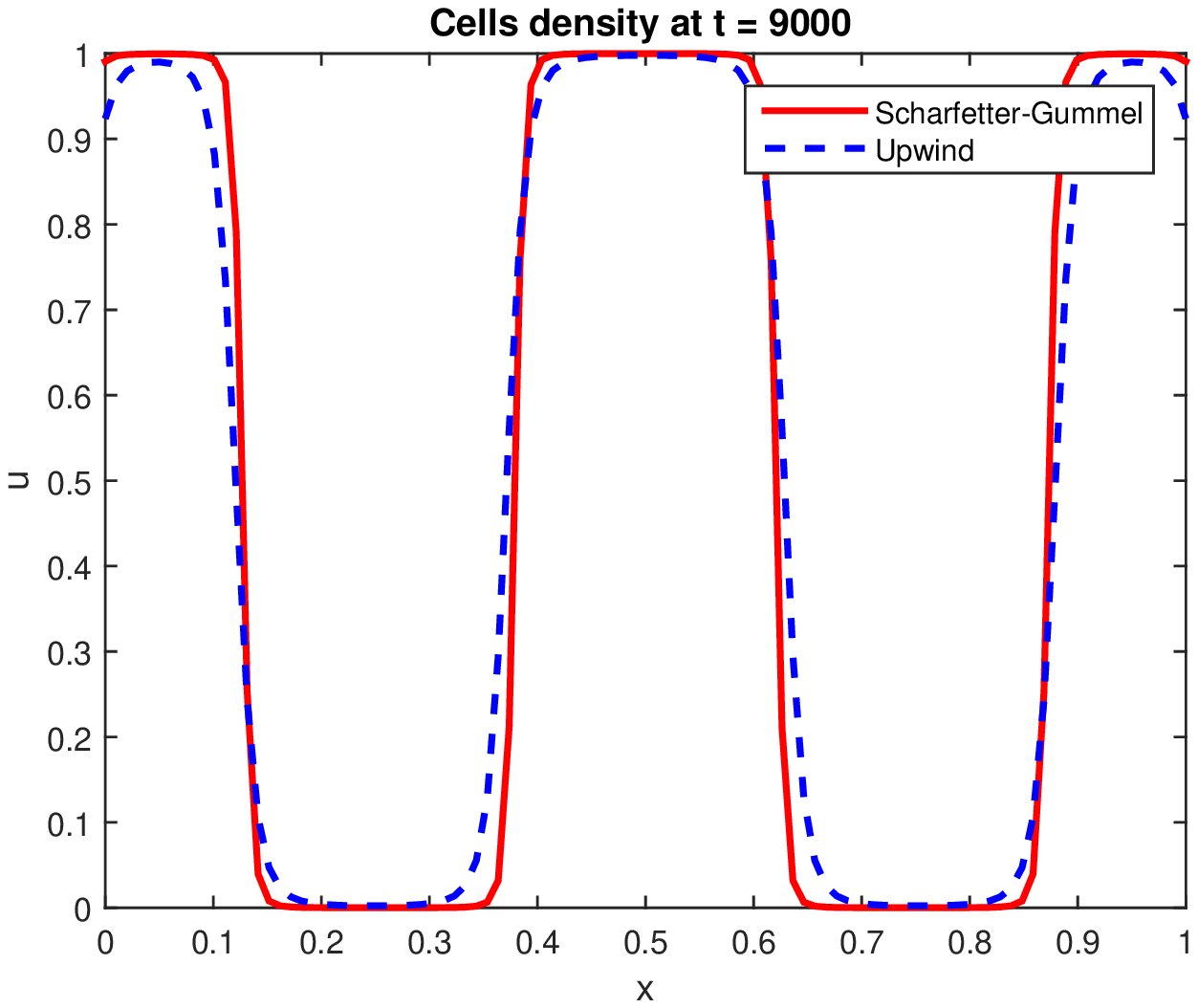}
		\caption{}
		\label{fig:logistic_SG_UP_9000}
	\end{subfigure}
	\begin{subfigure}[b]{0.32\textwidth}
		\includegraphics[width=\textwidth]{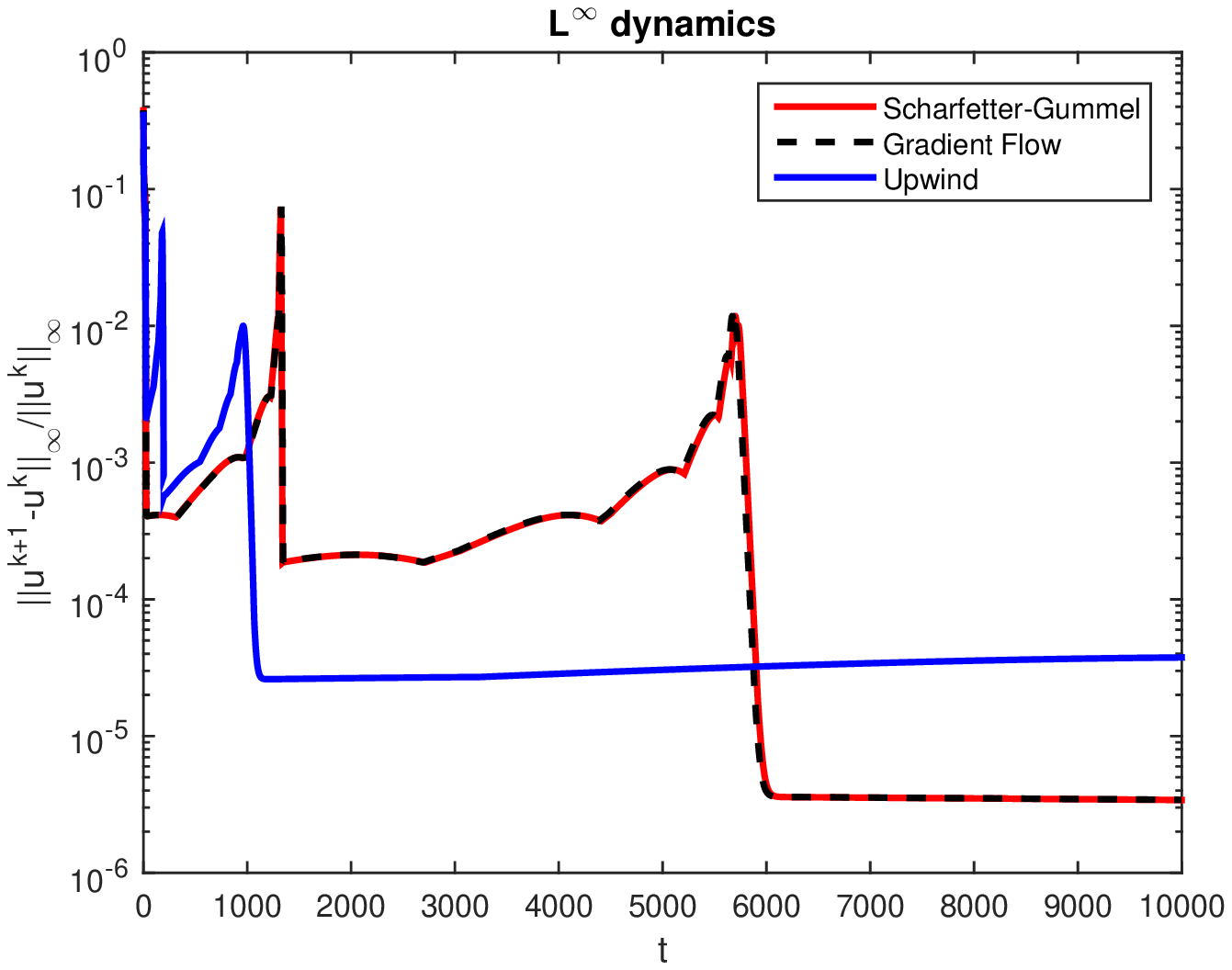}
		\caption{}
		\label{fig:logistic_conv}
	\end{subfigure}
	\caption{\textbf{Stationary profiles and dynamics.} \textsc{(A), (B)} Comparison of the stationary profiles of solutions to the Scharfetter-Gummel (red line) and the upwind (blue, dashed line) schemes at $t = 50$ and $t= 9000$. \textsc{(C)} Normalized $L^\infty$ variation for the three schemes.}
	\label{fig:logistic_SG_UP}
\end{figure}

\begin{figure}[htp]
	\centering
	\includegraphics[width=0.7\linewidth]{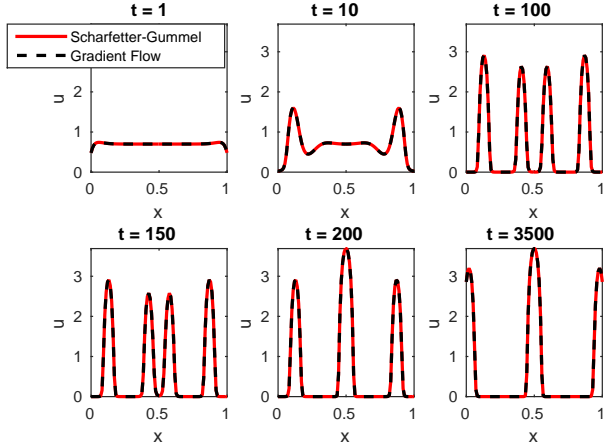}
	\caption{\textbf{Evolution in time of solutions} to~\ref{eq:ucoeff} in the exponential case $\vp(u) = u e^{-u}$ with $\chi / D = 24$. We solved the equation with the Scharfetter-Gummel (red line) and the gradient flow schemes (black, dashed line) with $I = 100$ and $\dt = 1$. As for the logistic model, the two schemes give the same solution.}
	\label{fig:exp_SG_GF_ev}
\end{figure}

As for the schemes, Figure~\ref{fig:logistic_SG_GF_ev} shows that the Scharfetter-Gummel and the gradient flow approaches give the same solution; no difference can be spotted. This is not true for the upwind approach. In Figure~\ref{fig:logistic_SG_UP_ev}, we compare the solutions to the Scharfetter-Gummel (red line) and the upwind (blue, dashed line) schemes. The upwind solution transitions faster from one metastable structure to the following than the Scharfetter-Gummel one. In fact, as proved above, the latter preserves discrete stationary profiles which, using the no-flux boundary conditions, solve the equation
\begin{equation}
\label{eq:stationary_profiles}
\f{\p u}{\p x} = \f{\chi}{D} \vp(u) \f{\p v}{\p x}.
\end{equation}

From~\ref{eq:stationary_profiles}, it is clear that, in the logistic case, the expected stationary solutions are 0-1 plateaus (or ``steps'') connected by a sigmoid curve which is increasing or decreasing when $v$ is. We refer again to \cite{Potapov2005} and also to \cite{Dolbeault2013} for a detailed study of the stationary solutions and their properties for the logistic Keller-Segel system. In Figures~\ref{fig:logistic_SG_UP_50} and~\ref{fig:logistic_SG_UP_9000}, we compare two stationary solutions to the Scharfetter-Gummel and upwind schemes, at time $t=50$ and $t = 9000$ respectively. The Scharfetter-Gummel scheme approximates the 0-1 plateaus of metastable solutions better than the upwind scheme, whose solutions have smoother edges.
We hypothesize that, since the Scharfetter-Gummel scheme preserves the metastable profiles better, it will also better preserve the time during which the solution will remain very close to a metastable state, up until the next transition. Consequently, this would mean that the upwind scheme accelerates the real dynamics. 

Moreover, in~\Cref{fig:logistic_conv} we compare the $L^{\infty}$ dynamics of the three schemes, computing the quantity $||u^n- u^{n-1}||_{\infty}/||u^{n-1}||_{\infty}$ for each $n$. The peaks shown by this figure correspond to the transitions from one profile to another. Observe that, for both solutions of the Scharfetter-Gummel and the gradient flow scheme, the two peaks are further away in time than the ones from the upwind scheme: this confirms that the upwind solution is in advance when it comes to transitioning. Nevertheless, from $t \approx 6000$, the relative errors of the upwind solution are consistently greater that the ones from the two other approaches, thus confirming that only the Scharfetter-Gummel and the gradient flow schemes better preserve the exact discrete stationary profiles as well as the metastable ones. Also notice that none of the schemes produce overshoot, due to our upwinding of the term in $\psi(u)$.
\\
\begin{figure}[htp]
	\centering
	\begin{subfigure}[b]{0.32\textwidth}
		\includegraphics[width=\textwidth]{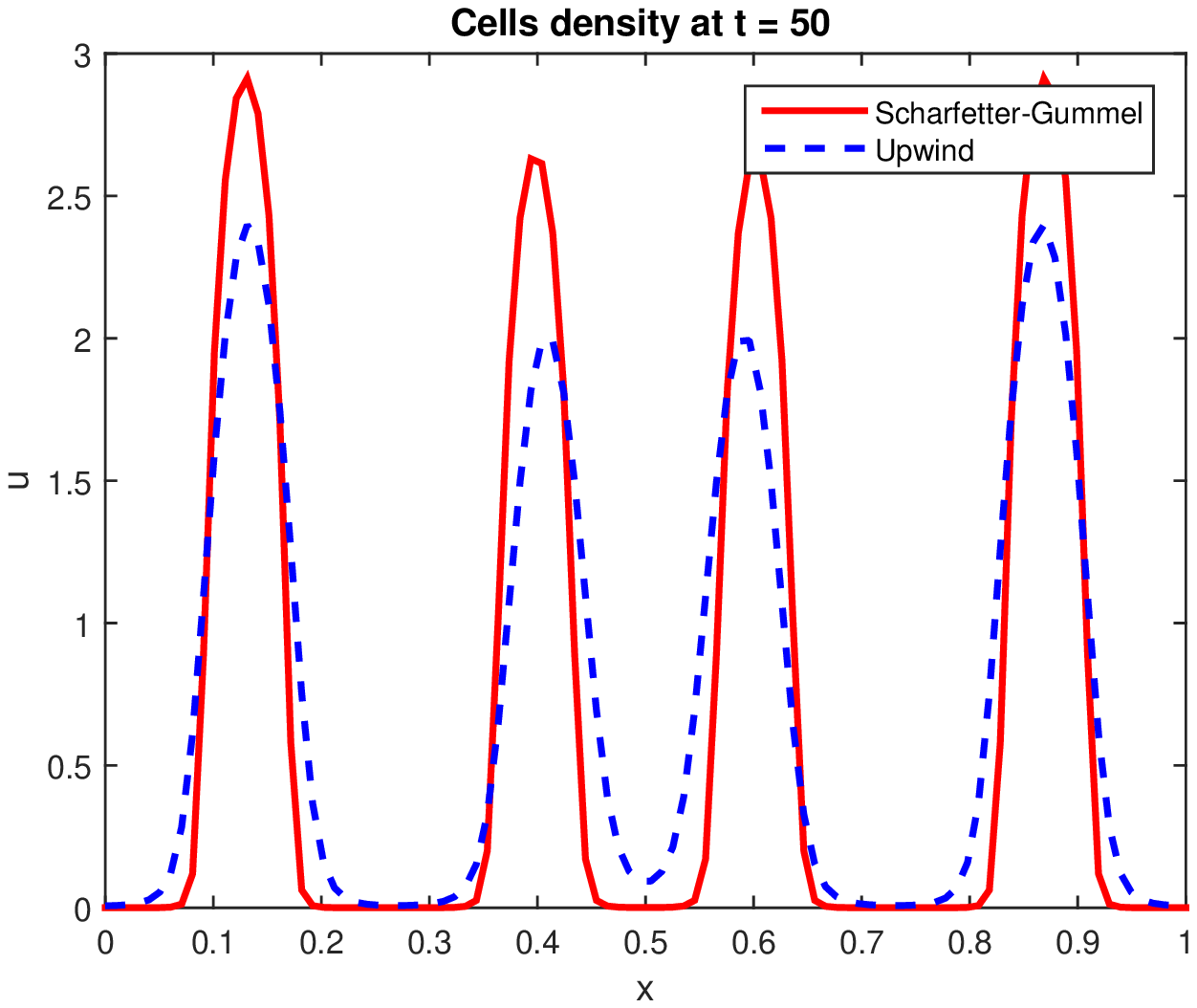}
		\caption{}
		\label{fig:exp_SG_UP_50}
	\end{subfigure}
	\begin{subfigure}[b]{0.32\textwidth}
		\includegraphics[width=\textwidth]{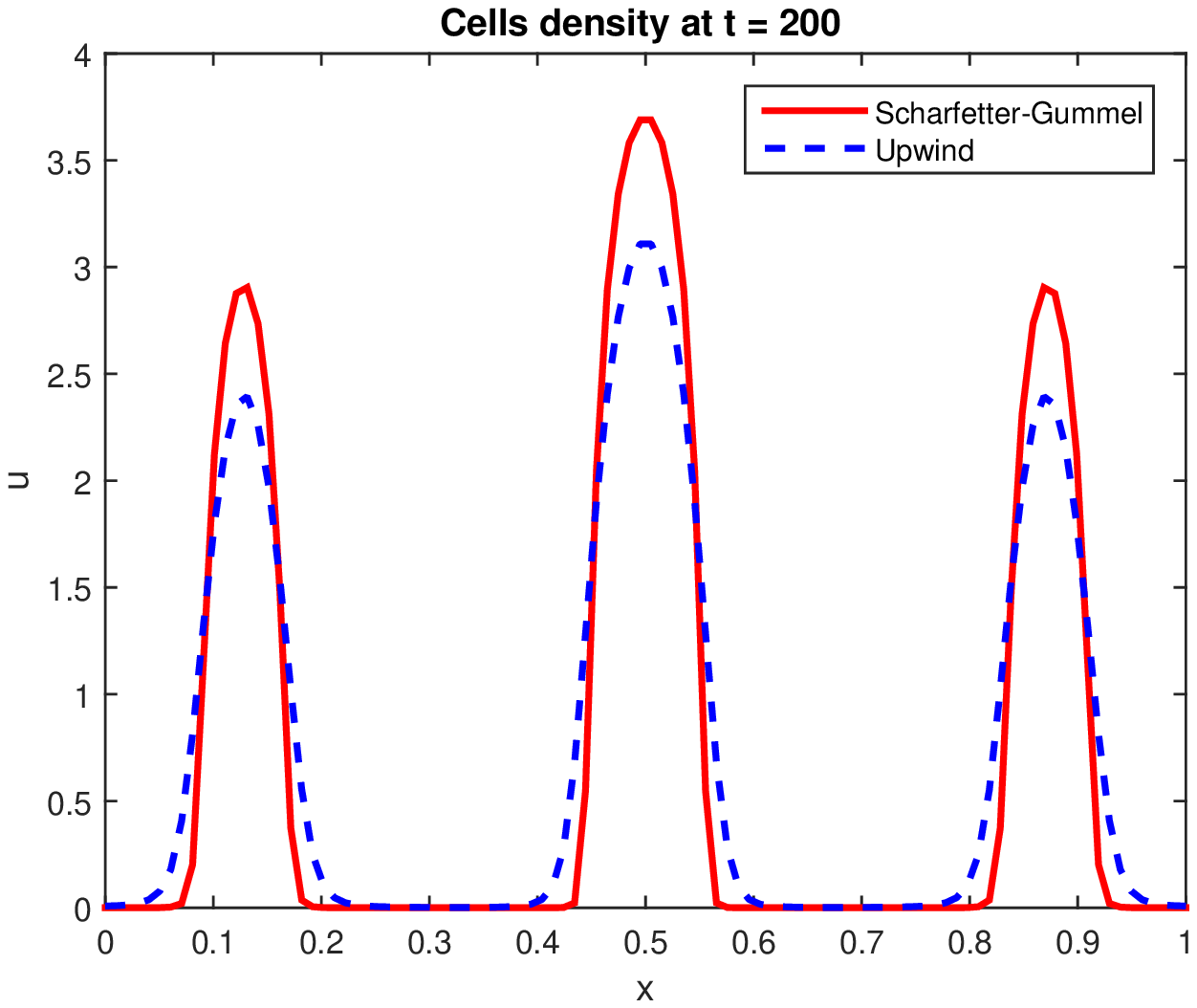}
		\caption{}
		\label{fig:exp_SG_UP_200}
	\end{subfigure}
	\begin{subfigure}[b]{0.32\textwidth}
		\includegraphics[width=\textwidth]{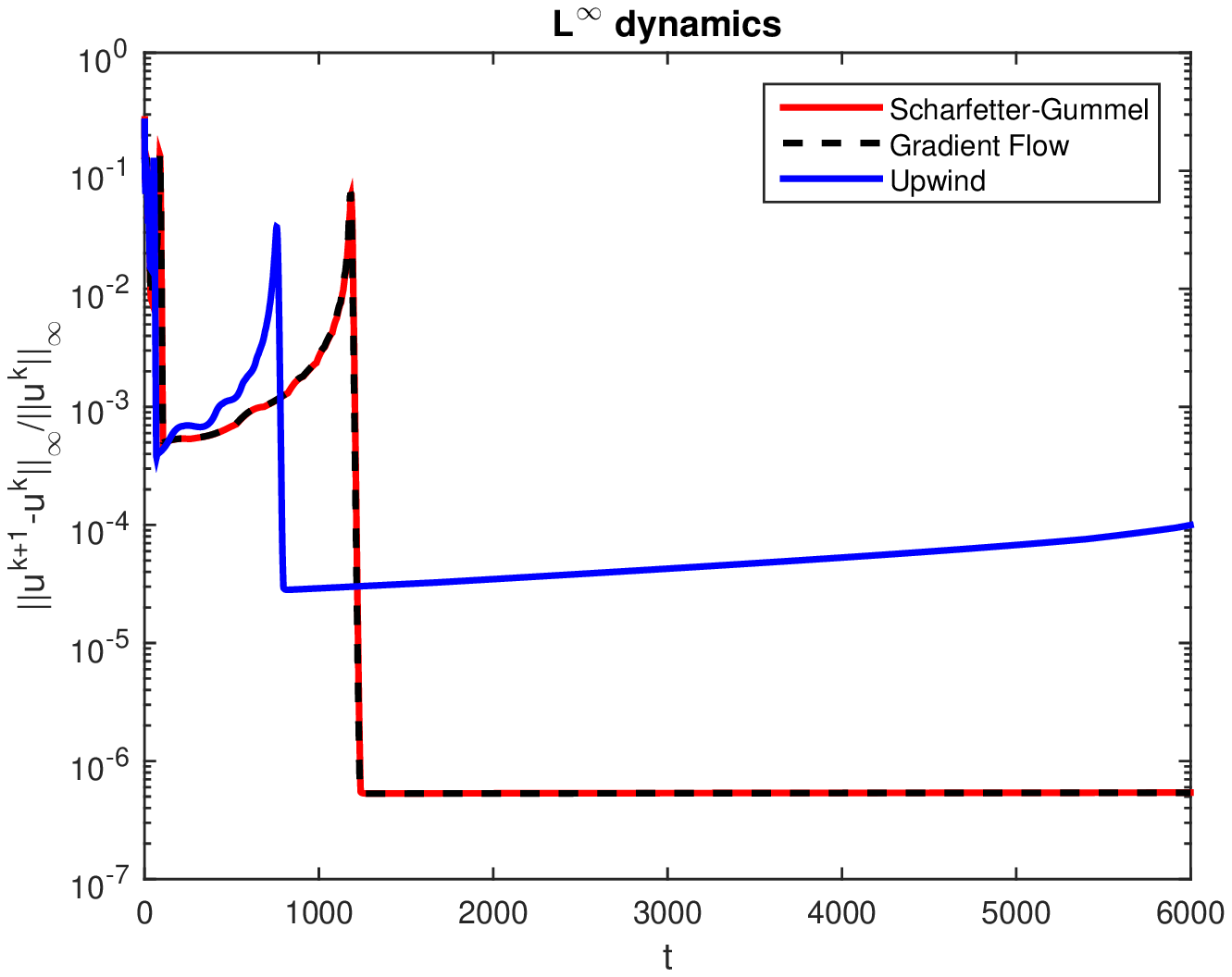}
		\caption{}
		\label{fig:exp_conv}
	\end{subfigure}
	\caption{\textbf{Stationary profiles and dynamics.} \textsc{(A), (B)} Comparison of the stationary profiles obtained with the Scharfetter-Gummel (red line) and the upwind scheme (blue, dashed line) at $t = 50$ (left) and $t= 200$. \textsc{(C)} Normalized $L^\infty$ variation for the three schemes.}
	\label{fig:ex_SG_UP}
\end{figure}

\begin{figure}[htp]
	\centering
	\includegraphics[width=0.7\linewidth]{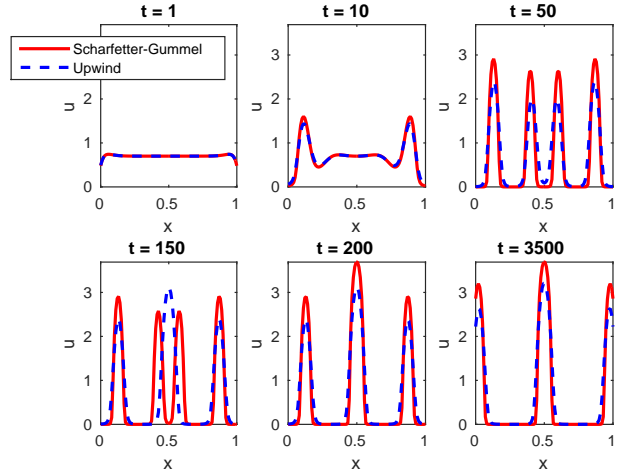}
	\caption{\textbf{Evolution in time of solutions} to \ref{eq:ucoeff} in the exponential case $\vp(u) = u e^{-u}$ with $\chi / D = 24$. We compare the solutions of the Scharfetter-Gummel (red line) and the upwind schemes (blue, dashed line) obtained with $I = 100$ and $\dt = 1$ for different times.}
	\label{fig:exp_SG_UP_ev}
\end{figure}

Next, we consider an exponentially decreasing form of the chemotactic sensitivity function with $\chi / D = 24$. Again, we take as initial condition a random spatial perturbation of the constant steady state $u^0 = 0.7$ and solve the equation on 100 equidistant points. The evolution in time of discrete solutions obtained with the three numerical approaches are compared in Figures~\ref{fig:exp_SG_GF_ev} and~\ref{fig:exp_SG_UP_ev}. In this model, no initial upper bound for the solution is imposed, so that the cells aggregate ``naturally'' where the chemoattractant has the greatest concentration, resulting in profiles without the plateaus observed in the logistic model. However, solutions face the same kind of transitions observed before, evolving from one stationary profile to another. As before, the Scharfetter-Gummel and the gradient flow approaches give the same solutions (\Cref{fig:exp_SG_GF_ev}), while the solution of the upwind scheme evolves faster. In Figures~\ref{fig:exp_SG_UP_50} and~\ref{fig:exp_SG_UP_200}, we compare stationary profiles obtained with the different approaches while in~\Cref{fig:exp_conv} we compare dynamics of the solutions. This last figure shows that, as for the logistic model, smaller errors can be expected for the Scharfetter-Gummel and gradient flow approaches when steady states are reached.

\section{Conclusion}

In the context of the Generalized Keller-Segel system, we have presented constructions of numerical schemes  which extend previous works~\cite{Carrillo2018, LiuWangZhou2018}, built on two different views of energy dissipation. Our construction unifies these two views, the gradient flow and Scharfetter-Gummel symmetrizations. Our schemes preserve desirable continuous properties: positivity, mass conservation, exact energy dissipation, discrete steady states. Being correctly tuned between implicit and explicit discretization, they can handle large time steps without CFL condition.

The present work is motivated by experiments of breast cancer cells put in a 3D structure mimicking the conditions they meet in vivo, namely in the extracellular matrix. After a few days, images of 2D sections show that cells have organized as spheroids, a phenomenon believed to be driven by chemotaxis. The spheroids can then be interpreted as Turing patterns for Keller-Segel type models and it is crucial to use appropriate schemes for them to be distinguishable from actual steady states or numerical artifacts. Comparing 2D simulations of such models with these experimental images will be the subject of future work.

In fact, it is important to remark that the schemes we presented here in 1D could be easily extended to rectangular domains, without loss of properties~\ref{prop1}--\ref{prop4}. However, it remains a perspective to treat more general geometries in a multi-dimensional setting with our approach.

\section*{Acknowledgments} The authors acknowledge partial funding from the ``ANR blanche" project Kibord [ANR-13-BS01-0004] funded by the French Ministry of Research. 
\\
B.P. has received funding from the European Research Council (ERC) under the European Union's Horizon 2020 research and innovation program (grant agreement No 740623).

\appendix

\section{Existence and uniqueness for monotone schemes}
\label{sec:existunique}

We recall sufficient conditions for which an implicit Euler discretization in time can be solved, independently of the step-size. This is the case for a monotone scheme. The proof relies on the existence of sub- and supersolutions, and thus  also yields the preservation of positivity and other pertinent bounds as we have used in Section~\ref{sec:discret}.  
\\ 
We consider the problem of finding a unique solution $\left(u_i^{n+1}\right)$ to the nonlinear equation arising in Section~\ref{sec:discret} which reads
\begin{equation}\label{eq:Au}
\frac{u_i^{n+1} -u_i^n}{\Delta t}+\f{1}{\dx}\big[ \underbrace {F (u_i^{n}, u_{i+1}^{n}, v_i^{n}, v_{i+1}^{n},  u_{i}^{n+1}, u_{i+1}^{n+1} )}_ 
{F_{i+\f12}^{n+1}}
- F_{i-\f12}^{n+1} \big] = 0,  \; i =1,\ldots,I.
\end{equation}

We write a general proof for a scheme of the form 
\begin{equation}\label{eq:A}
u_i+ A_{i+\f12} (u_{i},u_{i+1} )- A_{i-\f12} (u_{i-1},u_{i} ) = f_i,   \qquad i = 1, \ldots, I,
\end{equation}
where we consider the problem of finding a solution $(u_i)$ (which stands for $u_i^{n+1}$). 

Here we assume that the $f_i$ are given (it stands for $u^n_i$) and that the $A_{i+ \f12}$ are Lipschitz continuous and, a.e., 
\begin{equation}\label{A:monotone}
\p_1 A_{i+\f12}(\cdot,\cdot ) \geq 0, \qquad \p_2 A_{i+\f12}(\cdot,\cdot) \leq 0, \qquad i = 1, \ldots, I,
\end{equation}
and there are a supersolution $(\bar U_i)_{ i =1\; \ldots ,I}$  and a subsolution $(\underline U_i)_{ i =1\; \ldots ,I} $ such that for all $i =1,\ldots, I,$
\begin{align}\label{A:sub_super}
\bar U_i+ A_{i+\f12} (\bar U_i, \bar U_{i+1} ) - A_{i-\f12}  (\bar U_{i-1}, \bar U_i )\geq& f_i,\\ 
\underline U_i+ A_{i+\f12}( \underline U_i, \underline U_{i+1} )- A_{i-\f12} (\underline U_{i-1}, \underline U_i ) \leq& f_i.
\end{align}

We build a solution of~\ref{eq:A} using an evolution equation
\begin{equation}\label{eq:Adyn}
\f{du_i(t)}{dt} + u_i (t) + A_{i+\f12}( u_i(t), u_{i+1} (t))- A_{i-\f12}(u_{i-1}(t), u_i(t) ) = f_i,  \quad i = 1, \ldots, I.
\end{equation}
\begin{theorem} \label{th:monotone}
	Assume~\ref{A:monotone} and the existence of a subsolution and of a supersolution. Then, 
	\\[1mm]
	{\bf (i)} For a supersolution (resp. subsolution) initial data, the dynamics~\ref {eq:Adyn} satisfies $\f{d \bar u_i(t)}{dt} \leq 0$ (resp. $\f{d \underline u_i(t)}{dt} \geq 0$) for all times $t\geq 0$, and thus $\bar u_i(t)$ is a supersolution (resp. subsolution) for all times.
	\\[1mm]
	{\bf(ii)} A subsolution is smaller than a supersolution.
	\\[1mm]
	{\bf(iii)}  $\bar u_i(t)$ and $\underline u_i(t)$ converge to the same solution of \ref{eq:A}.
\end{theorem}

\begin{proof}
(i) We prove the statement with the supersolution. We set 
\begin{equation*}
z_i(t) = \f{d \bar u_i(t)}{dt}, \qquad z_i(0) \leq 0, \qquad i = 1, \ldots, I.
\end{equation*}
Since the $A_{i+\f12}$ are Lipschitz continuous, from equation \ref{eq:Adyn} we deduce that the quantities $\f{d \bar u_i(t)}{dt}$ are also Lipschitz continuous. From Rademacher's Theorem,  the $z_i$ are differentiable a.e. and their a.e. derivatives are also their distributional derivatives.

Differentiating equation~\ref{eq:Adyn}, we obtain for $ i =1,\ldots I$, and for a.e. $t>0$ 
\begin{equation*}
\f{dz_i(t)}{dt} + z_i (t)  + [\p_1 A_{i+\f12}   - \p_2 A_{i-\f12} ]\; z_{i} (t) = -\p_2 A_{i+\f12} \; z_{i+1} (t) +  \p_1 A_{i-\f12} \;  z_{i-1} (t). 
\end{equation*}
The solution cannot change sign and thus for $i =1,\ldots ,I$, $z_i(t) \leq 0$ for all times.
\\
(ii)
Consider $\underline u, \; \bar u$ sub (super) solutions. Set $w= \underline u- \bar u$ and we want to prove that $w\leq 0$. 

We write for $i =1,\ldots ,I$

\begin{equation*}
\begin{split}
w_i  &+ [A_{i+\f12} (\underline u_{i}, \underline u_{i+1} ) -   A_{i+\f12} (\bar u_{i}, \underline u_{i+1} ) ] +  [ A_{i+\f12} (\bar u_{i}, \underline u_{i+1} )-A_{i+\f12} (\bar u_{i}, \bar u_{i+1} )]
\\ &  - [A_{i-\f12} (\underline u_{i-1}, \underline u_{i} ) - A_{i-\f12} (\bar u_{i-1}, \underline u_{i}) ] - [ A_{i-\f12} (\bar u_{i-1}, \underline u_{i} )- A_{i-\f12} (\bar u_{i-1}, \bar u_{i} ) ] \leq 0.
\end{split}
\end{equation*} 
For $i =1,\ldots ,I$, we multiply by ${\rm sgn}_+(w_i):=1_{w_i >0}$ and add the relations to conclude that
\begin{equation*}
\sum_{i=1}^{I} (w_i)_+  + \sum_{i=1}^{I-1} J_{i+\f12} +  \sum_{i=1}^{I-1} K_{i+\f12} = 0, 
\end{equation*}
with  
\begin{align*}
J_{i+\f12} &=  [A_{i+\f12} (\underline u_{i}, \underline u_{i+1} ) -   A_{i+\f12} (\bar u_{i}, \underline u_{i+1} ) ] \; [ {\rm sgn}_+(w_i) - {\rm sgn}_+(w_{i+1})], \\
K_{i+\f12} &=  [ A_{i+\f12} (\bar u_{i}, \underline u_{i+1} )-A_{i+\f12} (\bar u_{i}, \bar u_{i+1} )] \; [ {\rm sgn}_+(w_i) - {\rm sgn}_+(w_{i+1})] .
\end{align*}
For each of the these terms, we show that $J_{i+\f12}  \geq 0$, $K_{i+\f12}  \geq 0$, as follows. Only the case when the + signs in the right brackets are different has to be considered. Assume for instance that 
\begin{equation*}
\underline u_i \geq \bar u_i, \qquad \text{and} \quad  \underline u_{i+1} \leq \bar u_{i+1}.
\end{equation*}
Then, we have by assumption \ref{A:monotone}, 
\begin{align*}
[A_{i+\f12}  &(\underline u_{i}, \underline u_{i+1} ) -   A_{i+\f12} (\bar u_{i}, \underline u_{i+1} ) ] \geq 0 \;  \Rightarrow J_{i+\f12} \;  \geq 0,\\
[ A_{i+\f12} &(\bar u_{i}, \underline u_{i+1} )-A_{i+\f12} (\bar u_{i}, \bar u_{i+1} )] \geq 0 \;  \Rightarrow K_{i+\f12} \;  \geq 0.
\end{align*}
Therefore $\sum_{i=1}^{I} (w_i)_+ \leq 0$ and this implies $w_i \leq 0$ for all $i$. From (i) and (ii) we infer that the subsolution increases and is bounded from above by the supersolution, so that it exists for all times. Similarly, the supersolution exists for all times and thus, so does the solution $u_i(t)$ to~\ref{eq:Adyn}, so that we can speak of its convergence.
\\
(iii) This is clear since the limits are solutions.~\end{proof}

{
	\bibliography{bibliography}
	\bibliographystyle{acm}}

\medskip
Received xxxx 20xx; revised xxxx 20xx.
\medskip

\end{document}